\documentclass[11pt]{article}
\usepackage{amsmath,amssymb,amsthm}
\usepackage{graphicx,subfigure,float,url}
\usepackage{pdfsync}
\usepackage[english]{babel}
\usepackage[utf8]{inputenc}

\usepackage{fancybox}
\usepackage{listings}
\usepackage{mathtools}
\usepackage{amsfonts}
\usepackage[bottom]{footmisc}
\usepackage{verbatim}

\usepackage{float}
\usepackage{makeidx}

\normalbaselines

\newtheorem{Theorem}{Theorem}[section]

\newtheorem{Definition}[Theorem]{Definition}

\newtheorem{Proposition}[Theorem]{Proposition}

\newtheorem{Lemma}[Theorem]{Lemma}

\newtheorem{Corollary}[Theorem]{Corollary}

\newtheorem{Remark}[Theorem]{Remark}

\newcommand{\N}{\mathbb N}

\newcommand{\RR}{{{\rm I} \kern -.15em {\rm R} }}

\newcommand{\C}{{{\rm l} \kern -.42em {\rm C} }}

\newcommand{\nat}{{{\rm I} \kern -.15em {\rm N} }}

\newcommand{\be}{\begin{equation}}
\newcommand{\ee}{\end{equation}}
\newcommand{\beq}{\begin{eqnarray}}
\newcommand{\eeq}{\end{eqnarray}}
\newcommand{\beqs}{\begin{eqnarray*}}
\newcommand{\eeqs}{\end{eqnarray*}}
\newcommand{\bt}{\begin{Theorem}}
\newcommand{\et}{\end{Theorem}}
\newcommand{\br}{\begin{Remark}}
\newcommand{\er}{\end{Remark}}
\newcommand{\bc}{\begin{Corollary}}
\newcommand{\ec}{\end{Corollary}}
\newcommand{\bl}{\begin{Lemma}}
\newcommand{\el}{\end{Lemma}}
\newcommand{\bd}{\begin{definition}}
\newcommand{\ed}{\end{definition}}

\renewcommand{\geq}{\geqslant}
\renewcommand{\leq}{\leqslant}

\title{Flocking estimates for the Cucker-Smale model with\\ time lag and hierarchical
leadership}
\author{
Cristina Pignotti\footnote{Dipartimento di Ingegneria e Scienze dell'Informazione e Matematica, Universit\`{a} di L'Aquila, Via Vetoio, Loc. Coppito, 67010 L'Aquila Italy (\texttt{pignotti@univaq.it}).}
\and Irene Reche Vallejo\footnote{MathMods Program, Università di L'Aquila, Via Vetoio, Loc. Coppito, 67010, L'Aquila, Italy (\texttt{irenereche92@gmail.com}).}
}

\date{}

\begin{document}

\textwidth=160 mm

\textheight=225mm

\parindent=8mm

\frenchspacing

\maketitle

\begin{abstract}
We analyze the Cucker-Smale model under hierarchical leadership in presence of a time delay. By using a Lyapunov functional approach and some induction arguments we will prove convergence to consensus for every positive delay $\tau .$ We also prove a flocking estimate in the case of a free-will leader. These results seem to point out the advantage of a hierarchical structure in order to contrast time delay effects that frequently appear in real situations.
\end{abstract}

\vspace{5 mm}

\def\qed{\hbox{\hskip 6pt\vrule width6pt
height7pt
depth1pt  \hskip1pt}\bigskip}



\section{Introduction}
\label{pbform}

\setcounter{equation}{0}
In recent years the study of collective behavior of autonomous agents has attracted an increasing interest in several scientific disciplines, e.g. ecology, biology, social sciences, economics, robotics (see
\cite{Axe, Aydogdu, Bullo, Camazine, CFTV, Couzin, Hel, MY, Parrish, Perea, Toner, TB}).
The Cucker-Smale model has been proposed and studied in \cite{CS1, CS2} as a model for flocking, namely for phenomena where autonomous agents reach a consensus based on limited environmental information.
Let us consider $N\in \N$ agents and let $(x_i(t), v_i(t))\in\RR^{2d},$ $i=1,\dots, N,$ be their phase-space coordinates. As usual $x_i(t)$ denotes the position of the $i^{\textrm{th}}$ agent and $v_i(t)$ the velocity.
The finite-dimensional Cucker-Smale model is, for $t>0,$
\begin{equation}\label{standard}
\begin{split}
\dot x_i(t)&= v_i(t),\\
\dot v_i(t)&=\sum_{j=1}^N a_{ij}(t)(v_j(t)-v_i(t)),\qquad i=1,\dots, N,
\end{split}
\end{equation}
where  the communication rates $a_{ij}(t)$ are of the form
\begin{equation}\label{potential}
a_{ij}(t)=\psi (\vert x_i(t)-x_j(t)\vert )\,,
\end{equation}
for a suitable positive non-increasing potential function
$\psi: [0, +\infty)\rightarrow (0, +\infty )\,.$

\noindent We define the diameters in space and velocity,
\begin{equation}\label{XV}
X(t):=\displaystyle { \max_{i,j}\vert x_i(t)-x_j(t)\vert}
\quad\mbox {\rm and}\quad
\displaystyle {V(t):= \max_{i,j}\vert v_i(t)-v_j(t)\vert}\,.
\end{equation}
\begin{Definition}{\rm
We say that a solution of (\ref{standard}) converges to consensus (or flocking) if
\begin{equation}\label{cons}
\sup_{t>0} X(t)<+\infty\quad\quad\mbox {\rm and}\quad \quad \lim_{t\rightarrow +\infty} V(t)=0\,.
\end{equation}
}\end{Definition}

The potential initially considered by Cucker and Smale in \cite{CS1, CS2} is $\psi(s)=\frac H
{(\sigma +s^2)^\beta }$ with $H, \sigma >0$ and $\beta\geq 0$. They proved that there  is
unconditional convergence to flocking whenever
 $\beta <1/2$.
If $\beta\geq 1/2$, they  obtained convergence to flocking under appropriate assumptions on the values of the initial variances on positions and speeds. Actually, unconditional flocking has been proved also in the case $\beta=\frac 12$  (see e.g. \cite{HL}).

The extension of the flocking result to the case of non symmetric communication rates has been proposed by Motsch and Tadmor \cite{MT}. Several other variants and generalization have been proposed including more general potentials, cone-vision constraints, leadership (see e.g. \cite{Couzin, CuckerDong, HaSlemrod, Mech, MT_SIREV, Shen, Vicsek, Yates}), stochastic terms (\cite{CuckerMordecki, Delay2, HaLee}), pedestrian crowds (see \cite{Cristiani, Lemercier}), infinite-dimensional kinetic models  (see \cite{Albi, Bellomo, canuto, carfor, degond, HT, Toscani}), topological models (\cite{BD, H}), control models
(see \cite{Borzi, Caponigro1, Caponigro2, PRT, Wongkaew}).

In this paper we consider the so called Cucker-Smale system with hierarchical leadership introduced by
Shen \cite{Shen}. In Shen's model
the agents are ordered in a specific way, depending on which other agents they are leaders of or led by. Indeed, it may happen in real situations, e.g. in animals groups, that some agents are more  influential than the others.
It is also natural to assume that  information from other agents is received after a certain time delay or that every agent needs a time to elaborate it.

Then, here, we are interested in the asymptotic analysis of the Cucker-Smale model with hierarchical leadership in presence of time delay effects. In particular, we assume that the agent $i,$ at time $t,$  changes its velocity  taking into account the information from any other agent $j$ at a previous time $t-\tau,$ i.e. $a_{ij}(t-\tau) v_j(t-\tau ),$ for a fixed positive time delay $\tau .$ This is of course a simplified model. Indeed, we have to mention that, physically, the time delay for transmission should depend on the distance between the agents. On the other hand, our analysis also holds, without substantial modifications, if the time delay is a bounded positive function $\tau (t)$ of the time variable $t$ (cfr. \cite{PT}).

For other extensions of Shen's results we refer to \cite{Da, Li, LiY, LX, LHX}.

\noindent
Before introducing our model we recall some definitions from \cite{Shen}.

\begin{Definition}
The \textbf{leader set} $\mathcal{L}(i)$ of an agent $i$ in a flock $\{1, 2,\dots, N\}$ is the subgroup of agents that directly influence agent $i$, i.e. $\mathcal{L}(i)=\{j\;|\;a_{ij}>0\}$.
\end{Definition}

The Cucker-Smale system considered by Shen is, for all $i=1, \dots, N$ and $t>0$,

\begin{equation}\label{CSShen}
\begin{array}{l}
\displaystyle{
\frac{d x_i}{dt}={v_i},}\\
\displaystyle{
\frac{d {v_i}}{dt}=\sum_{j\in\mathcal{L}(i)}a_{ij}(t)({v_j}-{v_i}).}
\end{array}
\end{equation}

\begin{Definition}
A flock $\{1, \dots, N\}$ is an \textbf{HL-flock}, namely a flock under hierarchical leadership, if the agents can be ordered in such a way that:
\begin{enumerate}
	\item if $a_{ij}\neq0$ then $j<i$, and
	\item for all $i>1$, $\mathcal{L}(i)\neq \varnothing$.
\end{enumerate}
\end{Definition}

Then, if $\{1, \dots, N\}$ is an HL-flock, $j\ge i$ implies $a_{ij}=0.$ Instead, if  $j<i,$ it results $a_{ij}>0$ if agent $j$ is in the leader set, $\mathcal{L}(i),$ of agent $i,$ otherwise it is $a_{ij}=0.$
In particular, $\mathcal{L}(1)=\varnothing ,$ namely the agent $1$ is the ultimate leader.

In \cite{Shen},
the interaction potential was defined as in the original work by Cucker and Smale, with $\sigma =1,$  namely

\begin{equation}\label{potShen}
a_{ij}(t)=\psi(|x_i-{x_j}|)=\frac{H}{(1+|{x_i}-{x_j}|^2)^\beta},\quad j\in \mathcal{L}(i),
\end{equation}
with $H$ and $\beta$ positive constants, and
convergence to consensus was  proved when $\beta <\frac 1 2.$ The flocking result has then been  extended to the case $\beta=\frac 1 2$ (see \cite{RX}).

We consider here a variant of Shen's model including a positive time delay $\tau,$ namely our system is the  following

\begin{equation}\label{CSShendelay}
\begin{array}{l}
\displaystyle{\frac{d {x_i}}{dt}(t)={v_i}(t),}\\
\displaystyle{
\frac{d {v_i}}{dt}(t)=\sum_{j\in\mathcal{L}(i)}a_{ij}(t-\tau )({v_j}(t-\tau)-{v_i}(t)),}
\end{array}
\end{equation}
for all $i=1, \dots, N$ and $t>0$, where

\begin{equation}\label{aggiunto}
a_{ij}(t)=\psi(|x_i(t)-{x_j}(t)|), \quad j\in \mathcal{L}(i),
\end{equation}
for some
nonnegative, non-increasing, continuous interaction function $\psi.$

As usual, since we deal with a delay model, the initial data are given for $s\in[-\tau ,0]$,

\begin{equation}\label{IC}
\begin{array}{l}
{x_i}(s)=x_i^0(s),\\
{v_i}(s)=v_i^0(s),
\end{array}
\end{equation}
for some continuous functions $x_i^0$ and $v_i^0,$ $i=1,\dots, N.$

We will prove a flocking result when $\psi$ has divergent tail, namely if
$\int_a^{+\infty} \psi(s) ds =+\infty\,,$
for some $a>0.$
Then, on the one hand, in the undelayed case, i.e.  $\tau = 0,$ we extend the result of Shen by considering a general potential of divergent tail. On the other hand, we prove flocking results in presence of time delay. Note that we do not require any  smallness conditions on the size of the time delay.
In particular, the model with hierarchical leadership seems more robust, if compared with the standard Cucker-Smale model, against time delay effects.
Indeed, delayed versions of the standard Cucker-Smale model have been recently considered (see \cite{Delay1, Delay2, Choi, PT}) and flocking results have been proved but only for particular potentials (\cite{Delay1, Choi} ) or under appropriate structural assumptions (\cite{Delay2, PT}). This fact suggests that the presence of leaders  is more efficient in  contrasting the time delay effects which naturally  appear in dynamics describing the motion of group of birds, animals or other agents.

The paper is organized as follows.
In section \ref{sect2} we give some preliminary properties of system (\ref{CSShendelay}), in particular we prove the positivity property for solutions to the scalar model associated  and, using this, the boundedness property for the velocities of solutions to (\ref{CSShendelay}).
In section \ref{sect3} we prove our main theorem, namely the flocking result for the system (\ref{CSShendelay}).
Finally, in section \ref{sect4} we consider the model under hierarchical leadership and a free--will leader and we prove convergence to consensus under some growth assumption on the acceleration of the free--will leader.

\section{Preliminary results} \label{sect2}
\setcounter{equation}{0}

Here, some general properties of the Cucker-Smale model (\ref{CSShendelay}) are established, which will be useful to prove the main result regarding the emergence of flocking behaviour.

We will need to define the $m$-th level leaders sets.

\begin{Definition}
 For each agent $i=1,\dots, N,$
we define
\begin{enumerate}
\item the {\bf m-th level leaders set}  of $i$, as

$$\mathcal{L}^0(i)=\{i\},\;\; \mathcal{L}^1(i)=\mathcal{L}(i),\;\; \mathcal{L}^m(i)=\bigcup_{j\in \mathcal{L}^{m-1}(i)} \mathcal{L}(j),\ \ m\in\N;$$

\item the set of
all leaders of $i$, direct or indirect, as

$$[\mathcal{L}](i)=\mathcal{L}^0(i)\;\cup\; \mathcal{L}^1(i)\;\cup\; \dots$$
\end{enumerate}
\end{Definition}
In particular, by definition of an HL-flock, $[\mathcal{L}](1)= {\mathcal{L}}^0(1)=\{ 1\}.$

\begin{Proposition}\label{positivityTh}
Let $({x_i},{v_i}), \ i=1, \dots, N,$  be a solution of the Cucker--Smale system under hierarchical leadership $(\ref{CSShendelay})$.
Let
us  consider the following system of ordinary scalar differential equations
\begin{equation}\label{positivitywithdelay}
	\begin{array}{l}
	\displaystyle{
\frac{d\eta_i}{dt}(t)=\sum_{j\in\mathcal{L}(i)}a_{ij}(t-\tau)(\eta_j(t-\tau)-\eta_i(t)), \;\;\; i=1,\dots,N,\; t>0,}\\
\displaystyle{
\eta_i(s)=\eta_i^0(s), \quad \;s\in[-\tau ,0],}
\end{array}
\end{equation}
where $\eta_i^0(\cdot),$ $i=1,\dots, N,$ are continuous functions and $a_{ij}(t),$ for $i=1, \dots, N$ and $j\in \mathcal{L}(i),$ are defined as in $(\ref{aggiunto}).$
If $\eta_i^0(s)\geq0$ for all $i=1, \dots, N$, and all $s\in[-\tau ,0]$, then $\eta_i(t)\geq0$ for all $i$ and $t>0$.
\end{Proposition}

\begin{proof}
Note that in system (\ref{positivitywithdelay}), if an agent $j$ is in $[\mathcal{L}](i)$, then $\eta_j$ is not influenced by agents outside of $[\mathcal{L}](i)$. It is sufficient, then, to prove the statement for the system (\ref{positivitywithdelay}) restricted to agents in $[\mathcal{L}](i)$, for each $i=1, \dots, N$.

We will  proceed by induction. Consider the first agent, agent 1. By definition of an HL-flock, $\mathcal{L}(1)=\varnothing$, which implies that

\begin{equation}
	\frac{d\eta_1}{dt}=0 \ \Rightarrow \quad \eta_1(t)=\eta_1(0)=\eta_1^0(0)\geq 0,\quad \forall\  t\geq0.
	\label{eta1}
\end{equation}
The equation for agent 2 will be, using (\ref{eta1}),

$$\frac{d\eta_2}{dt}=a_{21}(t-\tau)(\eta_1(t-\tau)-\eta_2(t))=a_{21}(t-\tau)(\eta_1(0)-\eta_2(t)).$$

Proceeding for contradiction, assume there exists some $\bar{t}>0$ such that $\eta_2(\bar{t})<0$, and let
us denote

$$t^*=\inf\{t>0\;|\;\eta_2(s)<0\ \mbox{\rm for}\ s\in (t,\bar {t})\,\}.$$

Then, by definition of $t^*$, $\eta_2(t^*)=0$ and
$\eta_2(s)<0$ for $s\in (t^*,\bar{t}).$
Using (\ref{eta1}),

$$\frac{d\eta_2}{dt}(s)=a_{21}(s-\tau)({\eta_1(0)}-{\eta_2(s)})\geq 0,\quad s\in [t^*, \bar {t}),$$

which contradicts the fact that $\eta_2(t)<0$ for $t\in (t^*, \bar t)$ and $\eta_2(t^*)=0.$ Hence, $\eta_2(t)\geq 0$ for all $t\ge 0$.

Now, as the induction hypothesis, assume that $\eta_i(t)\geq 0$ for all $t>0$ and for all $i\in\{1,\dots, k-1\}$.

The equation for agent $k$ will be

$$\frac{d\eta_k}{dt}=\sum_{j\in\mathcal{L}(k)}a_{kj}(t-\tau)(\eta_j(t-\tau)-\eta_k(t)).$$
By contradiction, assume there exists $\bar{t}>0$ such that $\eta_k(\bar{t})<0$ and let

$$t^*=\inf\{t>0\;|\;\eta_k(s)<0\ \mbox{\rm for}\ s\in (t,\bar {t})\,\}.$$

Then, $\eta_k(t^*)=0$ and $\eta_k(s)<0$ for $s\in (t^*, \bar{t}).$ We can use the induction hypothesis on the agents $j\in\mathcal{L}(k)\subseteq \{1, \dots, k-1\}$, so

$$\frac{d\eta_k}{dt}(s)=\sum_{j\in\mathcal{L}(k)}a_{kj}(s-\tau)({\eta_j(s-\tau)}-{\eta_k(s)})\geq 0,\quad s\in [t^*, \bar t)\,,$$
which gives a contradiction.

\noindent Therefore, we have proved that $\eta_i(t)\geq 0$ for all $i\in\{1, \dots, N\}.$
 \end{proof}

\begin{Proposition}\label{boundedness}
Let $\Omega$ be a convex and compact domain in $\RR^d$ and  let $({x_i},{v_i}),$ $i=1, \dots, N,$ be a solution of system $(\ref{CSShendelay})$. If ${v_i}(s)\in\Omega$ for all $i=1,\dots,N$ and $s\in[-\tau ,0]$, then ${v_i}(t)\in\Omega$ for all $i=1,\dots,N$ and $t>0$.
In particular, taking $\Omega=B_{D_0}(0)$ with
$$D_0=\max_{1\leq i\leq N}\max_{s\in[-\tau ,0]}|{v_i}(s)|,$$ then $\vert {v_i}(t)|\leq D_0$ for all $t>0$ and $i=1,\dots, N$.
\end{Proposition}

\begin{proof}
Let $\nu\in S^{n-1}$ be a unit vector and  let ${a}\in\RR^d$ be a given vector. Define $\eta_i=\nu\cdot({v_i}-{a})$. We now claim that if $\eta_i(s)\geq 0$ for all $i$ and $s\in[-\tau ,0]$ then $\eta_i(t)\geq0$ for all $i$ and $t>0$. Using (\ref{CSShendelay}),

\begin{align*}
\frac{d \eta_i}{dt}(t)&=\nu\cdot \frac{d {v_i}}{dt}(t)\\
&=\nu\cdot\Big(\sum_{j\in\mathcal{L}(i)}a_{ij}(t-\tau)({v_j}(t-\tau)-{v_i}(t))\Big)\\
&=\nu\cdot\Big(\sum_{j\in\mathcal{L}(i)}a_{ij}(t-\tau)[({v_j}(t-\tau)-{a})-({v_i}(t)-{a})]\Big)\\
&=\sum_{j\in\mathcal{L}(i)}a_{ij}(t-\tau)(\eta_j(t-\tau)-\eta_i(t)).
\end{align*}
Now we can use Proposition \ref{positivityTh} and we see that, indeed, the claim is true. Then the result follows arguing as in the undelayed case (see Th. 4.2 of \cite{Shen}).
\end{proof}

\section{The flocking result} \label{sect3}
\setcounter{equation}{0}

Using the above propositions, we are ready to prove
our main result, namely the existence of flocking solutions in a flock under hierarchical leadership satisfying the Cucker-Smale  system with delay (\ref{CSShendelay}). To do this, we combine arguments used to deal with the  undelayed models in both standard CS-system \cite{HL} and CS-system  with hierarchical leadership \cite{Shen}. New ingredients are needed in order to take into account the delay term.

\begin{Theorem}\label{flockingtheoremHLdelay}
Let $(x_i, v_i),$ $i=1, \dots, N,$ be a solution of the Cucker-Smale system under hierarchical leadership with delay $(\ref{CSShendelay})$ with initial conditions $(\ref{IC}).$ Assume
\begin{equation}\label{divergenttail}
\int_a^{\infty}\psi(s)\;ds=+\infty,
\end{equation}
for some $a>0$. Then,
\begin{equation}\label{flockingestimate}
V(t)=O(e^{-Bt}),
\end{equation}
with a constant $B>0$ depending only on the initial configuration and the parameters of the system.
\end{Theorem}

\begin{proof}
We will use induction on the number of agents in the flock.

Consider first a flock of 2 agents ${1,2}$. Since, by definition of an HL-flock, $\mathcal{L}(2)\neq \varnothing$, we must have $\mathcal{L}(2)=\{1\}$, i.e. $a_{21}>0$. Again by definition of an HL-flock, $a_{12}=0$. Then,
	
$$
\frac{d{v_1}}{dt}=0 \quad \Rightarrow \quad {v_1}(t)= {v_1}(0),\quad\forall\ t>0, $$
and
\begin{equation}\label{C1}
\frac{d {v_2}}{dt}=a_{21}(t-\tau)({v_1}(t-\tau)-{v_2}(t))=a_{21}(t-\tau)({v_1}(0)-{v_2}(t)),\quad t\ge\tau.
\end{equation}	
We now denote
\begin{equation}\label{C2}
{x}^2(t)={x_2}(t)-{x_1}(t)\quad\mbox{\rm  and}\quad {v}^2(t)={v_2}(t)-{v_1}(t).
\end{equation}
Then, from (\ref{C1}), we have

\begin{equation}\label{vt}
\frac{d {v^2}}{dt}=\frac{d {v_2}}{dt}-{\frac{d {v_1}}{dt}}=-a_{21}(t-\tau){v^2}(t),\quad t\ge\tau ,
\end{equation}
and thus
$$
\frac 12 \frac {d \vert v^2\vert^2} {dt}=-a_{21}(t-\tau )\vert v^2\vert^2\,,
$$
which gives
\begin{equation}\label{C3}
 \frac{d|{v^2}|}{dt}\leq -\psi\left(\vert {x_2}(t-\tau)-{x_1}(t-	 \tau)|\right)\;|{v^2}(t)|\,, \quad t\ge\tau\,.
\end{equation}
Therefore, from (\ref{C3}), we deduce that
 $|{v^2}(t)|$ is decreasing in time and so

\begin{equation}\label{decreasingv}
	|{v^2}(t)|\leq |{v^2}(\tau)|\,,\quad t\ge\tau\,.
\end{equation}
Note also that the equation for the position is
$$\frac{d {x^2(t)}}{dt}={v^2(t)}$$
which easily gives

\begin{equation}\label{C5}
 \left|\frac{d|{x^2(t)}|}{dt}\right|\leq |{v^2}(t)|\,,\quad t>0\,.
 \end{equation}
Now, observe that

$$
\begin{array}{l}
\displaystyle{{x_1}(t-\tau)-{x_2}(t-\tau)={x_1}(t)-{x_2}(t)+\int_t^{t-\tau}{({x_1}-{x_2})^\prime }(s)\;ds}\\
\hspace{2 cm}\displaystyle{
={x_1}(t)-{x_2}(t)-\int_t^{t-\tau}{v^2}(s)\;ds,}
\end{array}
$$
which, along with (\ref{decreasingv}), implies that

\begin{equation}\label{C4}
|{x_1}(t-\tau)-{x_2}(t-\tau)|\leq |{x_1}(t)-{x_2}(t)|+\vert {v^2}(\tau )|\tau=|{x^2}(t)|+|{v^2}(\tau )|\tau\,,
\quad t\ge 2\tau\,,
\end{equation}
with $x^2(t), v^2(t)$ defined in (\ref{C2}).
Using this inequality in (\ref{vt}) along with the fact that $a_{21}(t)=\psi(|{x^2(t)}|)$ with $\psi$ decreasing, we obtain

\begin{equation}\label{vt2}	
	\frac{d|{v^2(t)}|}{dt}\leq -\psi(|{x^2}(t)|+|{v^2}(\tau )|\tau)\; \vert {v^2}(t)|\,,\quad t\ge 2\tau\,.
\end{equation}
Consider now the Lyapunov functionals (cfr. \cite{HL})

\begin{equation}\label{Lyap2}
\mathcal{L}^2_{\pm}(t)=|{v^2(t)}|\pm\Phi(|{x^2}(t)|+|{v^2}(\tau )|\tau),
\end{equation}
where the function $\Phi$ is such that
$\Phi^\prime (r)=\psi(r),$ $r\in (0,+\infty ).$ From  (\ref{vt2}), we obtain

\begin{equation}\label{C6}
\begin{array}{l}
\displaystyle{
\frac{d\mathcal{L}^2_{\pm}}{dt}=
\frac{d|{v^2(t)}|}{dt}\pm \psi(|{x^2(t)}|+|{v^2}(\tau )|\tau )\frac{d|{x^2(t)}|}{dt}
}\\
\hspace{1 cm}
\displaystyle{
\leq -\psi(|{x^2(t)}|+|{v^2}(\tau )|\tau )|{v^2(t)}|\pm \psi(|{x^2(t)}|+|{v^2}(\tau )|\tau)\frac{d|{x^2(t)}|}{dt}}\\
\hspace{1 cm}\displaystyle{
=\psi(|{x^2(t)}|+|{v^2}(\tau )|\tau )\left(
\pm\frac{d|{x^2(t)}|}{dt} -|{v^2(t)}| \right )\leq 0\,,\quad t\ge 2\tau\,,}
\end{array}
\end{equation}
where we have used (\ref{C5}).
Hence, $\mathcal{L}^2_{\pm}(t)\leq \mathcal{L}^2_{\pm}(\tau )$, so

$$|{v^2}(t)|\pm \Phi(|{x^2}(t)|+|{v^2}(\tau )|\tau)\leq |{v^2}(\tau )|\pm \Phi(|{x^2}(\tau )|+|{v^2}(\tau )|\tau),$$

or

\begin{equation}\label{C7}
|{v^2}(t)|-|{v^2}(\tau )|\leq \pm \Phi(|{x^2}(\tau )|+|{v^2}(\tau )|\tau) \mp \Phi (|{x^2}(t)|+|{v^2}(\tau )|\tau)\,.
\end{equation}
From (\ref{C7}) we then deduce,

$$|{v^2}(t)|-|{v^2}(\tau )|\leq-\left|\int_{|{x^2}(\tau )|+\tau|{v^2}(\tau )|}^{|{x^2}(t)|+\tau |{v^2}(\tau )|}\psi(s)\;ds\right|\,.$$
In particular,

\begin{equation}\label{v0}
|{v^2}(\tau )|\geq \left|\int_{|{x^2}(\tau )|+\tau |{v^2}(\tau )|}^{|{x^2}(t)|+\tau |{v^2}(\tau )|}\psi(s)\;ds\right |\,.
\end{equation}
Since we have assumed that  $\psi$ satisfies (\ref{divergenttail}), then there must exist some $x_M^2\geq 0$ such that
\begin{equation}\label{C9}
|{v^2}(\tau )|=\int_{|{x^2}(\tau )|+\tau |{v^2}(\tau )|}^{x^2_M}\psi (s)
\,ds
\end{equation}
 which, along with (\ref{v0}), implies

\begin{equation}\label{C10}
|{x^2}(t )|+\tau |{v^2}(\tau )|\leq x^2_M,
\quad t\ge 2\tau\,,
\end{equation}
since $\psi$ is a nonnegative function. Now, using (\ref{vt2}) along with the fact that $\psi$ is non-increasing, we have

$$\frac{d|{v^2(t)}|}{dt}\leq -\psi(x^2_M)|{v^2}(t)|,\quad t\ge 2\tau\,,$$
and the Gronwall inequality gives us
\begin{equation}\label{C11}
|{v^2}(t)|\leq e^{-\psi(x^2_M)(t-\tau )}|{v^2}(\tau )|,\quad t\ge 2\tau\,.
\end{equation}

Now, for the flock of two agents ${1,2}$ we have, since ${v_1}(t)$ is constant for $t\ge 2\tau ,$

\begin{equation}\label{C12}
|{v_1}(t)-{v_2}(t)|=|{v_1}(t-\tau)-{v_2}(t)|=O(e^{-\psi(x_M^2)t}).
\end{equation}
We have also,

\begin{equation}\label{C13}
|{v_2}(t-\tau)-{v_1}(t)|\leq |{v_2}(t-\tau)-{v_2}(t)|+|{v_2}(t)-{v_1}(t)|\,.
\end{equation}
Observe that,

\begin{equation}\label{C14}
\begin{array}{l}
\displaystyle{
|{v_2}(t-\tau)-{v_2}(t)|= \left|\int_{t-\tau}^t {v_2}^\prime (s)\;ds\right|=\left|\int_{t-\tau}^t a_{21}(s-\tau)({v_1}(s-\tau)-{v_2}(s))\;ds\right|}\\
\hspace{1 cm}
\displaystyle{
\leq c\int_{t-\tau}^te^{-\psi(x^2_M)s}\;ds
\leq c\tau e^{-\psi(x^2_M)(t-\tau)}=c\tau e^{\psi(x^2_M)\tau}e^{-\psi(x^2_M)t}=O(e^{-\psi(x^2_M)t}),}
\end{array}
\end{equation}
so we have

\begin{equation}\label{C15}
|{v_2}(t-\tau)-{v_1}(t)|=O(e^{-\psi(x^2_M)t}).
\end{equation}

Then, the 2-flock {1,2} satisfies the estimates (\ref{C13}), (\ref{C14}) and (\ref{C15}). Moreover, of course,
$\vert v_1(t-\tau )-v_1(t)\vert = O(e^{-\psi(x^2_M)t}),$ being $v_1(t)$ constant for $t\ge 2\tau\,.$

We assume now, by induction, that analogous exponential estimates are satisfied
for a  flock of
 $l-1$ agents ${1, \dots, l-1}$ with $l>2$, i.e. there exists some constant $b>0$ such that $\forall\ i, j=1,\dots,l-1,$

\begin{align}
	&|{v_i}(t)-{v_j}(t)|=O(e^{-bt}),\label{inductionassumptiondelay1}\\
	&|{v_i}(t-\tau)-{v_j}(t)|=O(e^{-bt}).
	\label{inductionassumptiondelay2}
\end{align}

 Then, we want to prove the same kind of estimates  also for a flock with $l>2$ agents $\{1, \dots, l\}$.
This will prove our theorem.

 Define the average position and velocity of the leaders of agent $l$,

\begin{equation}\label{C16}
{\hat{x_l}}=\frac{1}{d_l}\sum_{i\in\mathcal{L}(l)}{x_i}(t)\;\;\; \mbox{and}\;\;\;{\hat{v_l}}=\frac{1}{d_l}\sum_{i\in\mathcal{L}(l)}{v_i}(t), \;\;\; d_l=\#\mathcal{L}(l),
\end{equation}
where $\# A$ denotes the cardinality of a set $A.$
Also, define
\begin{equation}\label{C17}
 {x^l(t)}={x_l(t)}-{\hat{x_l}(t)}\quad\mbox{\rm  and}\quad
{v^l(t)}={v_l(t)}-{\hat{v_l}(t)}.
\end{equation}
Then,

\begin{equation}\label{C18}
\frac{d {v^l}}{dt}=\frac{d {v_l}}{dt}-\frac{d {\hat{v_l}}}{dt}=\sum_{j\in\mathcal{L}(l)}a_{lj}(t-\tau)({v_j}(t-\tau)-{v_l}(t))-\frac{d {\hat{v_l}}}{dt}.
\end{equation}
By adding and subtracting $\sum a_{lj}(t-\tau ){\hat{v_l}}$ in (\ref{C18}) we get

\begin{equation}\label{dvdtdelay}	
\frac{d {v^l}}{dt}=\sum_{j\in\mathcal{L}(l)}a_{lj}(t-\tau)({\hat{v_l}}(t)-{v_l}(t))+\sum_{j\in\mathcal{L}(l)}a_{lj}(t-\tau)({v_j}(t-\tau)-{\hat{v_l}}(t))-\frac{d {\hat{v_l}}}{dt}.
\end{equation}
Using the induction hypothesis (\ref{inductionassumptiondelay2}), since $\mathcal{L}(i),\mathcal{L}(l)\subseteq \{1, \dots, l-1\}$,

\begin{equation}\label{epsilon1delay}
	\frac{d {\hat{v_l}}}{dt}=\frac{1}{d_l}\sum_{i\in\mathcal{L}(l)}\frac{d {v_i}}{dt}=\frac{1}{d_l}\sum_{i\in\mathcal{L}(l)}\sum_{j\in\mathcal{L}(i)}a_{ij}(t-\tau)({v_j}(t-\tau)-{v_i}(t))=O(e^{-bt}).
\end{equation}
Using once again the induction hypothesis  (\ref{inductionassumptiondelay2}),

\begin{equation}\label{C19}
	\sum_{j\in\mathcal{L}(l)}a_{lj}(t-\tau)({v_j}(t-\tau)-{\hat{v_l}}(t))
	= \frac 1 {d_l}\sum_{j\in\mathcal{L}(l)}a_{lj}(t-\tau)
	\sum_{i\in\mathcal{L}(l)} (v_j(t-\tau) -v_i(t))
	=O(e^{-bt}).
\end{equation}

Now, (\ref{dvdtdelay}) becomes

\begin{equation} \label{vtdelayy}
	\frac{d {v^l}}{dt}=-\sum_{j\in\mathcal{L}(l)}a_{lj}(t-\tau) {v^l}(t)+ O(e^{-bt}),
\quad t\ge\tau\,.	
\end{equation}
with
$$a_{lj}(t-\tau)=\psi(|{x_l}(t-\tau)-{x_j}(t-\tau)|).$$
Observe that for every $j\in \mathcal{L}(l)$ it results

\begin{equation}\label{C20}
\begin{array}{l}
\displaystyle{
|{x_l}(t-\tau)-{x_j}(t-\tau)|\le |{x_l}(t-\tau)-{\hat{x}_l}(t-\tau)|+|{x_j}(t-\tau)-{\hat{x}_l}(t-\tau)|}
\\\\
\hspace{4 cm}\displaystyle{
\leq |{x^l}(t-\tau)|+M_l,}
\end{array}
\end{equation}
for some $M_l>0$, due to the induction hypotheses. Therefore, since ${\psi}$ is non-increasing,

$$
{\psi}(|{x_l}(t-\tau)-{x_j}(t-\tau)|)\geq {\psi}\left((|{x^l}(t-\tau)|+
M_l)\right ),
$$
which, along with (\ref{vtdelayy}), implies

\begin{equation}\label{dvdelay2}
	\frac{d|{v^l}|}{dt}\leq -d_l \psi\left(|{x^l}(t-\tau)|+M_l\right)|{v^l}(t)|+ce^{-bt},\quad t\ge\tau\,.
\end{equation}

Define
$$D_0=2\max_{1\leq i\leq l}\max_{s\in[-\tau,0]}{v_i}(s).$$
From Proposition \ref{boundedness}, $|{v_i}(t)|\leq D_0/2$ for all $i$ and for all $t>0$, which implies

$$|{v^l}(t)|\leq \frac{1}{d_l}\sum_{j\in\mathcal{L}(l)}|{v_j}(t)-{v_l}(t)|\leq \frac{1}{d_l}\sum_{j\in\mathcal{L}(l)}D_0=D_0.$$
Then,
\begin{equation}\label{C21}
|{x^l}(t-\tau)|\leq |{x^l}(t)|+\tau D_0,\quad t\ge \tau\,,
\end{equation}
which, in (\ref{dvdelay2}), yields

\begin{equation}\label{vt3}
	\frac{d|{v^l}|}{dt}\leq -d_l \psi\left(|{x^l}(t)|+\tau D_0+M_l\right)|{v^l}(t)|+ce^{-bt}.
\end{equation}

Now consider the Lyapunov
functionals
\begin{equation}\label{Lyal}
\mathcal{L}^l_{\pm}(t)=|{v^l}(t)|\pm d_l\Phi\Big(|{x^l}(t)|+{\tau D_0+M_l}\Big),
\end{equation}
where as before $\Phi$ is a function such that  $\Phi^\prime (r)=\psi(r),$ $r\in (0,+\infty).$ Then, denoting $\tilde M= \tau D_0+M_l,$

\begin{equation}\label{C22}
\begin{array}{l}
\displaystyle{
	\frac{d\mathcal{L}^l_{\pm}}{dt}=\frac{d|{v^l}|}{dt}\pm d_l\psi\left(|{x^l}(t)|+\tilde{M}\right)\frac{d|{x^l}|}{dt}}\\
	\hspace{1 cm}\displaystyle{
	\leq -d_l\psi\left(|{x^l}(t)|+\tilde{M}\right)|{v^l}(t)|+ce^{-bt}\pm d_l\psi\left(|{x^l}(t)|+\tilde{M}\right)\frac{d|{x^l}|}{dt}}\\
	\hspace{1 cm}
\displaystyle{	
	= d_l\psi\left(|{x^l}(t)|+\tilde{M}\right)\left(\pm\frac{d|{x^l}|}{dt}-|{v^l}(t)|\right)+ce^{-bt}\leq ce^{-bt},\quad t\ge\tau\,,}
\end{array}
\end{equation}
since, from $\frac {d{x^l}}{dt} ={v^l}$, we have $\left|\frac{d\vert {x^l}|}{dt}\right|\leq|{v^l}(t)|$.

Thus, integrating (\ref{C22}) in $[\tau , t],$  we deduce

$$\mathcal{L}^l_{\pm}(t)-\mathcal{L}^l_{\pm}(\tau )\leq c\int_\tau ^te^{-bs}\;ds=\frac{c}{b}(e^{-b\tau}-e^{-bt})\leq \frac{c}{b},$$
which implies

$$|{v^l}(t)|-|{v^l}(\tau )|\leq \pm d_l\left(\Phi\left(\vert {x^l}(\tau )|+\tilde {M}\right)-\Phi\left(|{x^l}(t)|+\tilde{M}\right)\right)+\frac{c}{b},$$
namely
\begin{equation}\label{C23}
 |{v^l}(t)|-|{v^l}(\tau )|\leq -d_l\left|\int_{|{x^l}(\tau )|+\tilde{M}}^{|{x^l}(t)|+\tilde{M}}\psi(s)\;ds\right| +\frac{c}{b}.
\end{equation}
In particular, from (\ref{C23}), we have

\begin{equation}\label{v02}
	|{v^l}(\tau )|+\frac c b\geq d_l\left|\int_{|{x^l}(\tau )|+\tilde{M}}^{|{x^l}(t)|+\tilde{M}}\psi(s)\;ds\right|.
\end{equation}

Since $\int_a^{+\infty}\psi(s)\;ds=+\infty$, this implies the existence of a constant  $x^l_M>0$ such that

$$|{v^l}(\tau )|+\frac{c}{b}=d_l\int_{|{x^l}(\tau )|+\tilde{M}}^{x_M^l}\psi(s)\;ds ,$$
which, along with (\ref{v02}), gives

$$|{x^l}(t)|+\tilde{M}\leq x_M^l,\quad\forall\  t\ge\tau\,,$$
since $\psi$ is a nonnegative function.

Using this in (\ref{vt3}), we have

$$\frac{d|{v^l}|}{dt}\leq -d_l\psi(x^l_M)|{v^l}(t)|+ce^{-bt},$$
and therefore, from the
Gronwall inequality we obtain,
\begin{equation}\label{C24}
\vert v^l(t)\vert \le Ce^{-B^lt}\,,
\end{equation}
for suitable positive constants $C, B^l.$

Then, from (\ref{C24}) and the induction hypothesis (\ref{inductionassumptiondelay1}), for every $j\in {\mathcal L}(l),$ we have

\begin{equation}
	|{v_l}(t)-{v_j}(t)|\le \vert v_l(t)-\hat v_l(t)\vert +
	\vert \hat v_l(t) -v_j(t)\vert =O(e^{-Bt}).
	\label{firstbound}
\end{equation}
Now, to complete the induction argument, the only thing left to prove is that, for all $t>0$ and $i,j\in\{1, \dots,l\}$,

\begin{equation}
|{v_i}(t-\tau)-{v_j}(t)|=O(e^{-Bt}),
\label{finalbound}
\end{equation}
for a suitable positive constant $B.$

If $i,j\in\{1, \dots, l-1\}$, then (\ref{finalbound}) is true by (\ref{inductionassumptiondelay2}). Consider the case $i\in\{1, \dots, l-1\}$ and $j=l$. Then,

$$
|{v_i}(t-\tau )-{v_l}(t)|
\leq
|{v_i}(t-\tau )-{v_i}(t)|+\vert {v_i}(t)-{v_l}(t)|=O(e^{-Bt}),$$
by (\ref{inductionassumptiondelay2}) and (\ref{firstbound}), for suitable $B.$

Consider now the case where $i=j=l$. Then, using the previous case we see that

\begin{equation}
\label{C30}
\begin{array}{l}
\displaystyle{
|{v_l}(t-\tau)-{v_l}(t)|=\left|\int_{t-\tau}^t {v_l}^\prime (s)\;ds\right|=\left|\int_{t-\tau}^t \sum_{k\in\mathcal{L}(l)}a_{lj}(s-\tau)\left({{v_k}(s-\tau)-{v_l}(s)}\right)\;ds\right|}\\
\hspace{2 cm}\displaystyle{
\leq \bar{c}\int_{t-\tau}^te^{-Bs}\;ds= \bar{c}\tau e^{-B(t-\tau)}=\bar{c}\tau e^{B\tau}e^{-B t}=O(e^{-Bt}).}
\end{array}
\end{equation}

Also for the last case, where $j\in\{1, \dots, l-1\}$ and $i=l$, using (\ref{C30}) we have

$$|{v_l}(t-\tau)-{v_j}(t)|\leq |{v_l}(t-\tau)-{v_l}(t)|+|{v_l}(t)-{v_j}(t)|=O(e^{-Bt}),$$
by the previous case and (\ref{firstbound}). With this, we can say that (\ref{finalbound}) is satisfied for all $i,j\in \{1,\dots,l\}$ and the theorem is proved.
\end{proof}

\begin{Remark}\label{serve}
{\rm
Note that the estimate on the diameter in velocity $V(t)$ implies
$$\sup_{t>0}\ X(t)<+\infty\,,$$
therefore Theorem
\ref{flockingtheoremHLdelay} ensures convergence to consensus in the sense of Definition \ref{cons}\,.
}
\end{Remark}
\section{Flocking under a free-will leader} \label{sect4}
\setcounter{equation}{0}

Here we analyze the case in which the ultimate leader agent of the HL-flock  may have a free-will acceleration, instead of moving with constant velocity as in the previous section.
This reflects natural situations in which the flock is approached, for instance, by a predator and then the  leader takes off first or it changes the velocity in order to avoid
any danger.

The model describing such a situation is

\begin{equation}\label{agent1}
\begin{array}{l}
\displaystyle{\frac{d {x_1}}{dt}(t)={v_1}(t),}\\
\displaystyle{
\frac{d {v_1}}{dt}(t)=f(t),}
\end{array}
\end{equation}
where $f:[0,+\infty)\rightarrow \RR^d$ is a continuous integrable function, that is,
\begin{equation}\label{acceleration}
\Vert f\Vert_1=\int_0^{+\infty} \vert f(t)\vert \, dt <+\infty\,,
\end{equation}
for the motion of the free-will leader, and the Cucker-Smale model under hierarchical leadership for the other agents,
 namely
\begin{equation}\label{CSShendelayFW}
\begin{array}{l}
\displaystyle{\frac{d {x_i}}{dt}(t)={v_i}(t),}\\
\displaystyle{
\frac{d {v_i}}{dt}(t)=\sum_{j\in\mathcal{L}(i)}a_{ij}(t-\tau )({v_j}(t-\tau)-{v_i}(t)),}
\end{array}
\end{equation}
for all $i\in\{2, \dots, N\}.$
As usual, in order to solve the delay problem we need to assign the initial data on the time interval $[-\tau , 0],$ i.e.
\begin{equation}\label{ICFW}
\begin{array}{l}
{x_i}(s)=x_i^0(s),\\
{v_i}(s)=v_i^0(s),
\end{array}
\end{equation}
for some continuous functions $x_i^0$ and $v_i^0,$ for $i=1,\dots ,N.$

\begin{Definition}
We define the {\bf depth} $\Gamma$ of the HL-flock $\{1,\dots , N\}$ as  the maximum number of information passages needed to reach every agent starting from the free--will leader, i.e.
  $$\Gamma = \max\ \{\ \# [\mathcal{L}](i)\  :\  i=1,\dots ,N\,\}.$$
\end{Definition}

 Our theorem below extends and generalizes the flocking result proved by Shen for the undelayed case (see Th. 5.1 of \cite{Shen}).

\begin{Theorem}\label{flockingtheoremHLdelayFW}
Let $(x_i, v_i),$ $i=1, \dots, N,$ be a solution of the Cucker-Smale system under hierarchical leadership with delay  $(\ref{agent1})$--$(\ref{CSShendelayFW})$ with initial conditions
$(\ref{ICFW}).$
Assume that $(\ref{divergenttail})$
is satisfied and that the acceleration of the free-will leader satisfies
\begin{equation}\label{suf}
\vert f(t)\vert =O((1+t)^{-\mu }),
\end{equation}
for some $\mu > \Gamma-1\,,$ where $\Gamma$ is the depth of the flock.
Then, the flocking estimate
\begin{equation}\label{flockingestimateFW}
V(t)=O((1+t)^{-(\mu-\Gamma+1)}),
\end{equation}
holds true.
\end{Theorem}

\begin{Remark}\label{ultimo}
{\rm
In the statement of Theorem \ref{flockingtheoremHLdelayFW}, the
condition $\mu > \Gamma -1$ depends on the depth of the  flock.
This is natural because it corresponds to require less free will for the leader of a {\sl deep} flock, in order to have consensus.
Indeed, it may happen that the ultimate leader leads some agents only through a large number of intermediate agents, namely several information passages. So, the constraint on the decaying exponent of the leader's acceleration has to be appropriate for the size of the cloud in order to have the formation of a coherent flock.
}
\end{Remark}

\noindent
{\sl Proof of Theorem \ref{flockingtheoremHLdelayFW}.}
We argue by induction. First of all, look at the first agent, namely the free-will leader.
From (\ref{agent1}) we deduce
$$v_1(t)=v_1(0)+\int_0^t f(s)\, ds\,,$$
and so, being $f$ integrable,
\begin{equation}\label{O1}
\vert v_1(t)\vert \le\vert v_1(0)\vert +\Vert f\Vert_1=C_1\,,\quad \forall\ t\ge 0\,.
\end{equation}
Now, look at the 2-flock. As before, let us denote
$$v^2(t)=v_2(t)-v_1(t) \quad\mbox{\rm and }\quad x^2(t)=x_2(t)-x_1(t),\quad t\ge 0\,.$$
From (\ref{agent1}) and (\ref{CSShendelayFW})
\begin{equation}\label{O2}
\begin{array}{l}
\displaystyle{\frac {dv^2}{dt}=\frac {dv_2}{dt}-\frac {dv_1}{dt}
=a_{21}(t-\tau )(v_1(t-\tau )-v_2(t))-f(t)}\\
\hspace{1 cm}\displaystyle{
=a_{21}(t-\tau )(v_1(t)-v_2(t))-a_{21}(t-\tau )\int_{t-\tau }^t v_1^\prime (s)\, ds -f(t)}\\
\hspace{1 cm} \displaystyle{=-a_{21}(t-\tau )v^2(t)-a_{21}(t-\tau )\int_{t-\tau }^tf(s)\, ds -f(t)\,.
}
\end{array}
\end{equation}
Now, from (\ref{suf}), it is immediate to see that
\begin{equation}
\left\vert \max_{s\in [0,+\infty )} \psi (s)\int_{t-\tau }^t f(s)\, ds\right \vert +\vert f (t)\vert \le C(1+t)^{-\mu}\,, \quad t>0\,,
\end{equation}
for a suitable positive constant $C\,.$
Therefore, from (\ref{O2}) we deduce
\begin{equation}\label{O3}
\frac {d\vert v^2\vert }{dt}\le -a_{21}(t-\tau )\vert
v^2(t)\vert +C(1+t)^{-\mu}\le C(1+t)^{-\mu}\,,
\end{equation}
from which follows
\begin{equation}\label{O4}
\vert v^2(t)\vert \le \vert v^2(0)\vert +C\int_0^{+\infty} (1+t)^{-\mu}\, dt\le C_2,\quad \forall\ t\ge 0\,,
\end{equation}
for some constant $C_2>0.$
Since
$$x^2(t-\tau )=x^2(t)+\int_t^{t-\tau}v^2(s)\, ds\,,$$
from (\ref{O4}) we can then deduce
\begin{equation}\label{O5}
\vert x^2(t-\tau )\vert \le \vert x^2(t)\vert +\tau C_2\,.
\end{equation}
From (\ref{O5}) and the first inequality in (\ref{O3}), recalling that $\psi$ is non-increasing, we obtain
\begin{equation}\label{O6}
\frac {d\vert v^2\vert }{dt}\le -\psi (\vert x^2(t)\vert +\tau C_2 )\vert
v^2(t)\vert +C(1+t)^{-\mu}\,.
\end{equation}
Now, in order to find a bound for $\vert x^2(t)\vert ,$ we introduce the  functionals

\begin{equation}\label{Lyap2FW}
\mathcal{F}^2_{\pm}(t)=|{v^2(t)}|\pm\Phi(|{x^2}(t)|+
\tau C_2),
\end{equation}

where the function $\Phi$ is such that
$\Phi^\prime (r)=\psi(r),$ $r\in (0,+\infty ).$ From  (\ref{O6}), we obtain

\begin{equation}\label{O7}
\begin{array}{l}
\displaystyle{
\frac{d\mathcal{F}^2_{\pm}}{dt}=
\frac{d|{v^2(t)}|}{dt}\pm \psi(|{x^2(t)}|+\tau C_2 )\frac{d|{x^2(t)}|}{dt}
}\\
\hspace{1 cm}
\displaystyle{
\leq -\psi(|{x^2(t)}|+\tau C_2 )|{v^2(t)}|\pm \psi(|{x^2(t)}|+\tau C_2)\frac{d|{x^2(t)}|}{dt} + C(1+t)^{-\mu}}\\
\hspace{1.6 cm}\displaystyle{
=\psi(|{x^2(t)}|+\tau C_2 )\left(
\pm\frac{d|{x^2(t)}|}{dt} -|{v^2(t)}| \right )
+C(1+t)^{-\mu}}\\
\hspace{2.1 cm}\displaystyle{
\leq C(1+t)^{-\mu }}\,,\quad \quad t\ge 0\,,
\end{array}
\end{equation}
where we have used (\ref{C5}).
Hence,
$$\mathcal{F}^2_{\pm}(t)\leq \mathcal{F}^2_{\pm}(0 )+C\int_0^{+\infty} (1+t)^{-\mu }\, dt =\mathcal{F}^2_{\pm}(0 )+K ,$$
so

$$|{v^2}(t)|\pm \Phi(|{x^2}(t)|+\tau C_2)\leq
|{v^2}(0)|\pm \Phi(|{x^2}(0)|+\tau C_2) +K,$$
or, equivalently,

\begin{equation}\label{O8}
|{v^2}(t)|-|{v^2}(0 )|\leq \pm \Phi(|{x^2}(0 )|+\tau C_2 ) \mp \Phi (|{x^2}(t)|+\tau C_2 )+K \,.
\end{equation}

From (\ref{O8}) we then deduce,

$$|{v^2}(t)|-|{v^2}(0 )|\leq-\left|\int_{|{x^2}(0)|+\tau C_2}^{|{x^2}(t)|+\tau C_2}\psi(s)\;ds\right|+K\,.$$
In particular,

\begin{equation}\label{O9}
|{v^2}(0)|+K \geq \left|\int_{|{x^2}(\tau )|+\tau |{v^2}(\tau )|}^{|{x^2}(t)|+\tau |{v^2}(\tau )|}\psi(s)\;ds\right | \,.
\end{equation}

Since we have assumed that  $\psi$ is positive and satisfies (\ref{divergenttail}), then there must exist some $x_R^2\geq 0$ such that

\begin{equation}\label{O10}
|{x^2}(t )|+\tau C_2 \leq x^2_R,
\quad t\ge 0\,.
\end{equation}
Now, using (\ref{O6}) along with the fact that $\psi$ is non-increasing, we have

$$\frac{d|{v^2(t)}|}{dt}\leq -\psi(x^2_R)|{v^2}(t)|+ C(1+t)^{-\mu} ,\quad t\ge 0\,,$$
and so, for every $T>0,$ the Gronwall's lemma implies
\begin{equation}\label{O11}
\begin{array}{l}
\displaystyle{\vert v^2(T)\vert \le e^{-\psi(x^2_R)\frac T2}\vert v^2( T/2 )\vert +\int_{\frac T2}^Te^{-\psi(x^2_R)(T-t)}\frac C {(1+t)^\mu }\, dt}\\
\hspace {1 cm} \displaystyle{\le e^{-\psi(x^2_R)\frac T 2}C_2
+O((1+T)^{-(\mu -1)})\,.}
\end{array}
\end{equation}
Thus,
\begin{equation}\label{O12}
\vert v_2(t)-v_1(t)\vert =O((1+t)^{-(\mu -1)})\,.
\end{equation}
Note also that
\begin{equation}\label{O13}
\vert v_1(t-\tau) -v_1(t)\vert \le \int_{t-\tau}^{+\infty} \vert f (t)\vert\, dt =O((1+t)^{-(\mu -1)})\,,
\end{equation}
and
\begin{equation}\label{O14}
\begin{array}{l}
\displaystyle{
\vert v_2(t-\tau) -v_2(t)\vert
\le \vert v_2(t-\tau ) -v_1(t-\tau )\vert }\\
\hspace{1.8 cm}\displaystyle{
+
\vert v_1(t-\tau ) -v_1(t)\vert +
\vert v_1(t) -v_2(t)\vert =O((1+t)^{-(\mu -1)})}\,.
\end{array}
\end{equation}
Therefore, (\ref{O12})--(\ref{O14}) imply
\begin{equation}\label{2flockFW}
\vert v_i(t-\tau) -v_j(t)\vert = O((1+t)^{-(\mu -1)}),\quad \mbox{\rm for}\ \ i,j\in \{1, 2\}\,.
\end{equation}

Now, as induction hypothesis, assume that
for a flock of
 $l-1$ agents $\{1, \dots, l-1\}$ with $2<l\le N$, we have
\begin{align}
	&|{v_i}(t)-{v_j}(t)|=O((1+t)^{-(\mu -l+2)}),\label{inductionassumptiondelay1FW}\\
	&|{v_i}(t-\tau)-{v_j}(t)|=O((1+t)^{-(\mu -l+2)})\,,
	\label{inductionassumptiondelay2FW}
\end{align}
for all $i,j\in \{ 1,\dots, l-1\}.$

 Then, we want to prove  the same kind of estimates  for a $l$ flock with $l$ agents.
This will prove our theorem.

As before, we use the average position and velocity of the leaders of agent $l$, introduced in (\ref{C16}) and let $x^l, v^l$  as in  (\ref{C17}).
Then,
as before we can write

\begin{equation}\label{rinomino}	
\frac{d {v^l}}{dt}=\sum_{j\in\mathcal{L}(l)}a_{lj}(t-\tau)({\hat{v_l}}(t)-{v_l}(t))+\sum_{j\in\mathcal{L}(l)}a_{lj}(t-\tau)({v_j}(t-\tau)-{\hat{v_l}}(t))-\frac{d {\hat{v_l}}}{dt}.
\end{equation}
Using the induction hypothesis (\ref{inductionassumptiondelay2FW}), since $\mathcal{L}(i),\mathcal{L}(l)\subseteq \{1, \dots, l-1\}$,

\begin{equation}\label{epsilon1delayFW}
	\frac{d {\hat{v_l}}}{dt}=\frac{1}{d_l}\sum_{i\in\mathcal{L}(l)}\frac{d {v_i}}{dt}=\chi_{1\in\mathcal{L}(l)}\frac{1}{d_l}{f}(t)+\frac{1}{d_l}\sum_{i\in\mathcal{L}(l)\backslash \{1\}}\frac{d {v}_i}{dt}=O((1+t)^{-(\mu -l+2)}).
\end{equation}
Using again the induction hypotheses  (\ref{inductionassumptiondelay2FW}),

\begin{equation}\label{O19}
\begin{array}{l}
\displaystyle{
	\sum_{j\in\mathcal{L}(l)}a_{lj}(t-\tau)({v_j}(t-\tau)-{\hat{v_l}}(t))}
\\ \hspace{1 cm}	
	\displaystyle{
	= \frac 1 {d_l}\sum_{j\in\mathcal{L}(l)}a_{lj}(t-\tau)
	\sum_{i\in\mathcal{L}(l)} (v_j(t-\tau) -v_i(t))
	=O((1+t)^{-(\mu -l+2)}).}
	\end{array}
\end{equation}

Therefore, (\ref{rinomino}) becomes

\begin{equation} \label{vtdelayyFW}
	\frac{d {v^l}}{dt}=-\sum_{j\in\mathcal{L}(l)}a_{lj}(t-\tau) {v^l}(t)+ O((1+t)^{-(\mu -l+2)}),
\quad t\ge 0\,.	
\end{equation}
with
$$a_{lj}(t-\tau)=\psi (\vert x_i(t-\tau )-x_l(t-\tau )\vert).$$
Then, arguing as in the proof of Theorem \ref{flockingtheoremHLdelay},
we arrive at

\begin{equation}\label{dvdelay2FW}
	\frac{d|{v^l}|}{dt}\leq -d_l \psi\left(|{x^l}(t-\tau)|+R_l\right)|{v^l}(t)|+ O((1+t)^{-(\mu -l+2)}),\quad t\ge 0\,,
\end{equation}
for a suitable positive constant $R_l\,.$

Note that (\ref{dvdelay2FW}) implies

\begin{equation}\label{M1}
\frac{d|{v^l}|}{dt}\leq C (1+t)^{-(\mu -l+2)},
\end{equation}
for a suitable constant $C.$ Then, from (\ref{M1}) we deduce

\begin{equation}\label{M2}
\vert v^l(t)\vert \le \vert v^l(0)\vert +C\int_0^{+\infty }
(1+t)^{-(\mu -l+2)}\, dt \le C_l,
\end{equation}
where we used that, since $\mu >N-1$ and $l\le N,$ the integral is convergent.

Then,
\begin{equation}\label{M3}
|{x^l}(t-\tau)|\leq |{x^l}(t)|+\int_{t-\tau }^t \vert v^l(s)\vert \, ds \le |{x^l}(t)|+ C_l\tau\,,
\quad t\ge 0\,,
\end{equation}
which, used in (\ref{dvdelay2FW}), yields

\begin{equation}\label{M4}
	\frac{d|{v^l}|}{dt}\leq -d_l \psi\left(|{x^l}(t)|+R_l+C_l\tau \right)|{v^l}(t)|+C (1+t)^{-(\mu -l+2)}.
\end{equation}

Now consider the functionals
\begin{equation}\label{LyalFW}
\mathcal{F}^l_{\pm}(t)=
|{v^l}(t)|\pm d_l\Phi\Big(|{x^l}(t)|+R_l+\tau C_l\Big),
\end{equation}
with, as before, $\Phi$  primitive function of $\psi\,.$
Using (\ref{M4}) we then obtain
\begin{equation}\label{M5}
\frac{d\mathcal{F}^l_{\pm}}{dt}\le C (1+t)^{-(\mu -l+2)}\,.
\end{equation}
Since the function in the right--hand side of (\ref{M5}) is integrable, we can prove that there exists a positive constant $x_R^l$ such that
$$\vert x^l(t)\vert +R_l+\tau C_l \le x_R^l,\quad t\ge 0\,.$$
We can thus restate (\ref{M4})
as
\begin{equation}\label{M6}
\frac{d|{v^l}|}{dt}\leq -d_l \psi\left( x^l_R\right )|{v^l}(t)|+C (1+t)^{-(\mu -l+2)}.
\end{equation}
Now we can apply the Gronwall's lemma as for the $2-$flock
case obtaining
\begin{equation}\label{K1}
\vert v^l(t)\vert =O ((1+t)^{-(\mu-l+1)})\,.
\end{equation}
Then, from (\ref{K1}) and the induction hypothesis (\ref{inductionassumptiondelay1FW}), for every $j\in {\mathcal L}(l),$ we have

\begin{equation}
	|{v_l}(t)-{v_j}(t)|\le \vert v_l(t)-\hat v_l(t)\vert +
	\vert \hat v_l(t) -v_j(t)\vert =O ((1+t)^{-(\mu-l+1)}).
	\label{firstboundFW}
\end{equation}
Now, to complete the induction argument, we only have to prove  that, for all $i,j\in\{1, \dots,l\}$,

\begin{equation}
|{v_i}(t-\tau)-{v_j}(t)|=O ((1+t)^{-(\mu-l+1)}).
\label{finalboundFW}
\end{equation}

If $i,j\in\{1, \dots, l-1\}$, then (\ref{finalboundFW}) is true by (\ref{inductionassumptiondelay2FW}). Consider the case $i\in\{1, \dots, l-1\}$ and $j=l$. Then,

$$
|{v_i}(t-\tau )-{v_l}(t)|
\leq
|{v_i}(t-\tau )-{v_i}(t)|+\vert {v_i}(t)-{v_l}(t)|=O ((1+t)^{-(\mu-l+1)}),$$
by (\ref{inductionassumptiondelay2FW}) and (\ref{firstboundFW}).

Consider now $i=j=l$. Then, using the previous case we see that

\begin{equation}
\label{C30FW}
\begin{array}{l}
\displaystyle{
|{v_l}(t-\tau)-{v_l}(t)|=\left|\int_{t-\tau}^t {v_l}^\prime (s)\;ds\right|=\left|\int_{t-\tau}^t \sum_{k\in\mathcal{L}(l)}a_{lk}(s-\tau)\left({{v_k}(s-\tau)-{v_l}(s)}\right)\;ds\right|}\\
\hspace{1.2 cm}\displaystyle{
\leq C\int_{t-\tau}^t(1+s)^{-(\mu -l+1)}\;ds\le C\tau (1+ t-\tau)^{-(\mu -l+1)}=O ((1+t)^{-(\mu-l+1)})\,.}
\end{array}
\end{equation}
Also for the last case, where $j\in\{1, \dots, l-1\}$ and $i=l$, using (\ref{C30FW}) we have

$$|{v_l}(t-\tau)-{v_j}(t)|\leq |{v_l}(t-\tau)-{v_l}(t)|+|{v_l}(t)-{v_j}(t)|=O ((1+t)^{-(\mu-l+1)}),$$
where we have used (\ref{firstboundFW}). With this, we can conclude  that (\ref{finalboundFW}) is satisfied for all $i,j\in \{1,\dots,l\}$ and the theorem is proved.\qed

\begin{Remark}\label{FWmigliora}
{\rm
The flocking result for the Cucker-Smale model under hierarchical leadership and with a free-will leader has been first obtained by Shen \cite{Shen}, under the same assumption (\ref{suf}) on the acceleration of the ultimate leader.
We extend her result by including delay effects. Moreover, we deal with a more general potential of interaction $\psi$ with respect to \cite{Shen}
where, indeed, the author deals with the potential of the former papers of Cucker and Smale (\cite{CS1, CS2}), $\psi (s)= \frac {H}{(1+s^2)^\beta},$ under the assumption $\beta <\frac 1 2\,.$
}
\end{Remark}

\noindent {\bf Acknowledgements.}
This work has been completed when the second author was at DISIM, University of L'Aquila, for the MathMods master program. The research of the first author is partially supported by the GNAMPA 2017 
project {\em
Comportamento asintotico e controllo di equazioni di evoluzione non lineari} (INdAM).
We would like to thank DISIM, MathMods  and GNAMPA-INdAM for the support.

\end{document}